\documentclass{llncs}
\usepackage[utf8]{inputenc}
\usepackage[pdftex]{graphicx}
\usepackage[pdftex]{color}
\usepackage{lmodern}
\usepackage{bbold}
\usepackage{amssymb} % that is for real numbers etc.
\usepackage{mathtools}
\usepackage{amsfonts}
\usepackage[shortlabels]{enumitem}
\usepackage{natbib}
\usepackage{subcaption}
\usepackage[margin=1cm]{caption}
\usepackage{dsfont}
\usepackage{pgfplots}
\usepackage{tikz}
\usepackage{filecontents}
\newcommand{\iid}{\operatorname{\stackrel{i.i.d.}{\sim}}}

\newcommand{\tr}{\mbox{\rm trace}}
\DeclareMathOperator*{\argmin}{\mbox{\rm argmin}}
\newtheorem{Def}{Definition}
\newtheorem{Th}[Def]{Theorem}

\newtheorem{Conj}[Def]{Conjecture}
\newtheorem{As}[Def]{Assumptions}

\newcommand{\cD}{\mathcal{D}}

\newcommand{\cN}{\mathcal{N}}

\newcommand{\SSS}{{{\mathbb S}^m}}

\newcommand{\Prb}{\mathbb P}

\newcommand{\inD}{\operatorname{\stackrel{\cD}{\to}}}

\newcommand{\mun}{\widehat{\mu}_n}

\newcommand{\EE}{\mathbb{E}}
\newcommand{\NN}{\mathbb{N}}
\newcommand{\RR}{\mathbb{R}}
\newcommand{\V}{\mathbb{V}}
\newcommand{\SSm}{\mathbb{S}^m}
\newcommand{\SSo}{\mathbb{S}^1}
\newcommand{\fm}{\mathfrak{m}}

\begin{document}
% \pagenumbering{gobble} 
\title{Finite Sample Smeariness on Spheres} 
\titlerunning{FSS on Spheres} 
% \markboth{Dimension Reduction on Polyspheres with Application to Skeletal Representations}{Dimension Reduction on Polyspheres with Application to Skeletal Representations}
\thispagestyle{empty}
% \begin{flushleft}
% \LARGE\bfseries Dimension Reduction on Polyspheres with  Application to Skeletal Representations\\[1cm]
% \end{flushleft}
\author{Benjamin Eltzner\inst{1}, Shayan Hundrieser\inst{2}  \and  Stephan Huckemann\inst{1}}
\authorrunning{Eltzner and Huckemann}
\institute{$^1$Georg-August-Universit\"at G\"ottingen, Germany, Felix-Bernstein-Institute for Mathematical Statistics in the Biosciences, \\
  $^2$Georg-August-Universit\"at G\"ottingen, Germany, Institute for Mathematical Statistics, \\
  Acknowledging DFG HU 1575/7, DFG GK 2088, DFG EXC 2067, DFG SFB 1465 and the Niedersachsen Vorab of the Volkswagen Foundation}

\maketitle
\thispagestyle{plain}

\begin{abstract}
    \emph{Finite Sample Smeariness} (FSS) has been recently discovered. It means that the distribution of sample Fr\'echet means of underlying rather unsuspicious random variables can behave as if it were smeary for quite large regimes of finite sample sizes. In effect classical quantile-based statistical testing procedures do not preserve nominal size, they reject too often under the null hypothesis. Suitably designed bootstrap tests, however, amend for FSS. On the circle it has been known that arbitrarily sized FSS is possible, and that all distributions with a nonvanishing density feature FSS. These results are extended to spheres of arbitrary dimension. In particular all rotationally symmetric distributions, not necessarily supported on the entire sphere feature FSS of Type I. While on the circle there is also FSS of Type II it is conjectured that this is not possible on higher-dimensional spheres.
\end{abstract}

%key words: non-Euclidean principal component analysis, medical imaging, data-driven metric

\section{Introduction}
% {\bf{Es fehlen wohl hier noch einige Referenzen...}}
In non-Euclidean statistics, the Fr\'echet mean \citep{F48} takes the role of the expected value of a random vector in Euclidean statistics. Thus an enormous body of literature has been devoted to the study of Fr\'echet means and its exploitation for descriptive and inferential statistics \citep{HL98,BP05, H_Procrustes_10,LeBarden2014,BL17}. For the latter, it was only recently discovered that the asymptotics of Fr\'echet means may differ substantially from that of its Euclidean kin \citep{HH15,EltznerHuckemann2019}. Initially, such examples were rather exotic. Corresponding distributions have been called \emph{smeary}. More recently, however, it has been discovered that also for a large class of classical distributions (e.g. all with nonvanishing densities on the circle, like, e.g. all von-Mises-Fisher distributions) Fr\'echet means behave in a regime up to a considerable sample sizes as if they were smeary. We call this effect \emph{finite sample smeariness} (FSS), also the term \emph{lethargic means} has been suggested. Among others, this effect is highly relevant for asymptotic one- and two-sample tests for equality of means. In this contribution, after making the new terminology precise, we illustrate the effect of FSS on statistical tests concerning the change of wind directions in the larger picture of climate change. 
 
 Furthermore, while we have shown earlier that FSS of any size can be present on the circle and the torus, here we show that FSS of arbitrary size is also present on spheres of arbitrary dimension, at least for local Fr\'echet means. For such, on  high dimensional spheres, distributions supported by barely more than a geodesic half ball may feature arbitrary high FSS. Moreover, we show that a large class of distributions on spheres of arbitrary dimension, namely all rotationally symmetric ones, e.g. all Fisher distributions, feature FSS. This means not only that the finite sample rate may be wrong, also the rescaled asymptotic variance of Fr\'echet means may be considerably different from the sample variance in tangent space.   
% {\bf gibt's dazu noch ein Beispiel? SMAH: Es gab doch in dieser Hinsicht Beobachtungen in dem Paper \"uber Smeariness auf Sph\"aren, oder?}

\section{Finite Sample Smeariness on Spheres}

Let $\SSm$ be the unit sphere in $\RR^{m+1}$ for $m > 1$ and $\SSo = [-\pi,\pi)/\sim$ with $-\pi$ and $\pi$ identified be the unit circle, with the distance
$$d(x,y) = \left\{\begin{array}{rcl}\arccos (x^Ty)&\mbox{ for }& x,y\in \SSm,\\ % \mbox{ and } d(x,y) = 
\min\{|y-x|, 2\pi- |y-x|\}&\mbox{ for }&x,y \in \SSo.
                  \end{array}\right.$$
For random variables $X_1,\ldots,X_n \iid X$ on $\SSm$ , $m\geq 1$, with silently underlying probability space $(\Omega,\Prb)$ we have the \emph{Fr\'echet functions}
\begin{eqnarray}\label{eq:Frechet-fcns} F(p) = \EE[d(X,p)^2]&\mbox{ and }& F_n(p) = \frac{1}{n}\sum_{j=1}^n d(X_j,p)^2\mbox{ for }p\in \SSm\,.\end{eqnarray}
% We also write $F^X$ and $F^X_n$ to refer to the underlying $X$. 
We work under the following assumptions. In particular, the third Assumption below is justified by \cite[Lemma 1]{TranEltznerHuGSI2021}.

\begin{As}\label{As:1}
Assume that
\begin{enumerate}
  \item $X$ is not a.s. a single point, 
  \item there is a unique minimizer $\mu = \argmin_{p\in \SSm} F(p)$, called the \emph{Fr\'echet population mean}, 
  \item for $m>1$, $\mu$ is the north pole $(1,0,\ldots,0)$ and $\mu =0$ on $\SSo$, 
  \item $\mun \in \argmin_{p\in \SSm} F_n(p)$ is a selection from the set of minimizers uniform with respect to the Riemannian volume, called a \emph{Fr\'echet sample mean},
\end{enumerate}
\end{As}

Note that $\Prb\{X = - \mu\} =0$ for $m>1$ and $\Prb\{X = - \pi\} =0$ on $\SSo$ due to \cite{LeBarden2014,HH15}.

\begin{Def}
We have the \emph{population variance} 
$$V := F(\mu) = f(0) = \EE[d(X,\mu)^2]\,,$$
which, on $\SSo$ is just the classical variance $\V[X]$, and 
the \emph{Fr\'echet sample mean variance} 
$$V_n :=  \EE[d(\mun,\mu)^2] $$ giving rise to the \emph{modulation}  
\begin{eqnarray*}
 \fm_n &:=& \frac{nV_n}{V}\,.
\end{eqnarray*} 
\end{Def}

We have the following finding from \cite{HundrieserEltznerHuckemann2020}

\begin{Th}\label{Th:ModulationProperties}
 Consider $X_1,\ldots,X_n \iid X$ on $\SSo$ and suppose that  $J \subseteq \SSo$ is the support of $X$. Assume Assumption \ref{As:1} and let $n>1$.  

 Then $\fm_n = 1$ under any of the two following conditions 
\begin{itemize}
 \item[(i)] $J$ is strictly contained in a closed half circle,
 \item[(ii)] $J$ is a  closed half circle and one of its end points is assumed by $X$ with zero probability.
\end{itemize}
Further, $\fm_n>1$ under any of the two following conditions 
\begin{itemize}
 \item[(iii)] the interior of $J$ contains a closed half circle,
 \item[(iv)] $J$ contains two antipodal points, each of which is assumed by $X$ with positive probability. 
\end{itemize}
Finally, suppose that $X$ has near $-\pi$ a continuous density $f$.  
\begin{itemize}
 \item[(v)] If $f(-\pi) =0$ then $\lim_{n\to \infty} \fm_n =1$,
 \item[(vi)] if $0<f(-\pi)<\frac{1}{2\pi}$ then $\lim_{n\to \infty} \fm_n =\frac{1}{(1-f(-\pi) 2\pi)^2} >1$\,.
 \end{itemize}
 
\end{Th}

In \cite{HH15} it has been shown that $f(-\pi) 2\pi$ can be arbitrary close to $1$, i.e. that $\lim_{n\to \infty} \fm_n$ can be arbitrary large. In fact, whenever 
$f(-\pi) 2\pi=1$, then  $\lim_{n\to \infty} \fm_n=\infty$.
These findings give rise to the following. % definition.

\begin{Def}
 We say that $X$ is
 \begin{itemize}
  \item[(i)] \emph{Euclidean} if $\fm_n = 1$ for all $n \in \NN$,
  \item[(ii)] \emph{finite sample smeary} if $1 <\sup_{n\in\NN} \fm_n < \infty$,
  \begin{itemize}
   \item[($ii_1$)]  \emph{Type I finite sample smeary} if $\lim_{n\to \infty} \fm_n >1$, 
   \item[($ii_2$)]  \emph{Type II finite sample smeary} if $\lim_{n\to \infty} \fm_n =1$,
   \end{itemize}
  \item[(iii)] \emph{smeary} if $\sup_{n\in\NN} \fm_n = \infty$.
%   \item[(iv)] \emph{$r$-power smeary} if (\ref{eq:Frecht-fcn}) holds with $r>2$.
 \end{itemize}
\end{Def}

\begin{figure}[t!]
\centering
\includegraphics[width = \textwidth, trim = 0 0 0 0, clip]{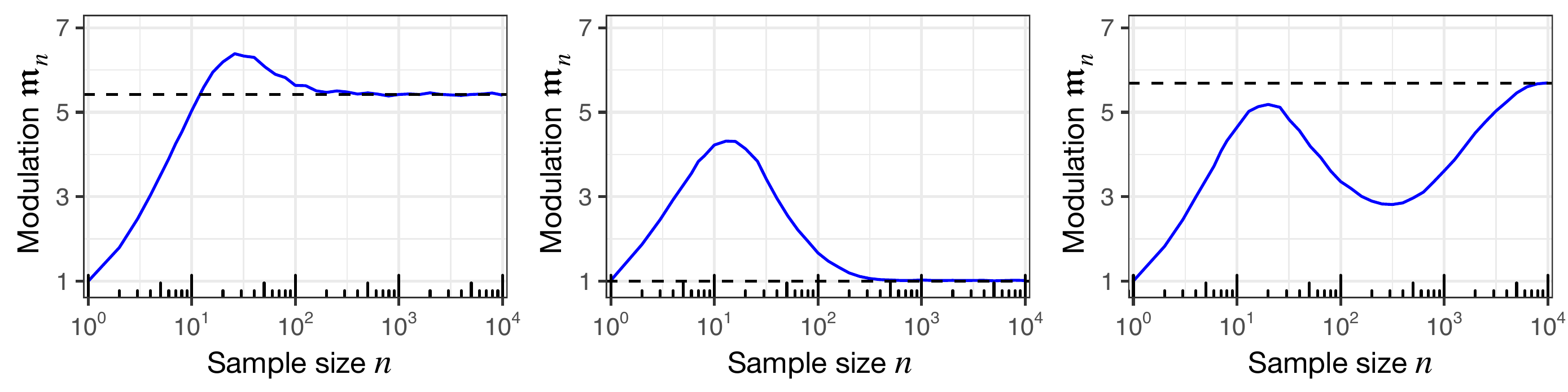}
 \caption{\it Modulation $\fm_n$ for von Mises distribution  (Mardia \& Jupp, 2000) with mean $\mu = 0$ and concentration $\kappa = 1/2$ (left), conditioned on $[-\pi+0.2,\pi-0.2]$ (center), and conditioned on $[-\pi,-\pi +0.1] \cup [-\pi+0.2,\pi +0.2]\cup [\pi-0.1,\pi)$ (right). The dashed lines represent the respective limits of $\fm_n$ obtained by Theorem \ref{Th:ModulationProperties} (v), (vi).  }\label{fig:circular-modulation-curves}
\end{figure}

\section{Why is Finite Sample Smeariness called Finite Sample Smeariness?}

Under FSS on the circle in simulations we see typical shapes of modulation curves in Figure \ref{fig:circular-modulation-curves}. For statistical testing, usually building on smaller sample sizes, as detailed further in Section \ref{scn:testing}, the initial regime is decisive, cf. Figure \ref{fig:circular-modulation-scheme}:

  There are constants $C_+, C_-, K > 0$, $0 < \alpha_- < \alpha_+ < 1 $ and integers $1 < n_- < n_+ < n_0$ satisfying $C_+ n_{-}^{\alpha_+}\leq C_- n_{+}^{\alpha_-}$, such that
  \begin{itemize}
    \item[(a)] $\forall n \in [n_-, n_+] \cap \mathbb{N} \, : \quad 1 < C_- n^{\alpha_-} \le \fm_n \le C_+ n^{\alpha_+}$.
    \item[(b)] $\forall n \in [n_0, \infty) \cap \mathbb{N} \, :~~ \quad \fm_n \le K $.
%     \item[(a)] $\forall n \in (n_-, n_+] \cap \mathbb{N} \, : \quad \textnormal{Var}[X]\leq C_- n^{\alpha_-} \le n \textnormal{Var}[\widehat{\mu}_n] \le C_+ n^{\alpha_+}$.
%     \item[(b)] $\forall n \in (n_0, \infty) \cap \mathbb{N} \, : \quad \textnormal{Var}[\widehat{\mu}_n] \le K n^{-1}$.
  \end{itemize}
  \begin{figure}[b!]
  \begin{center}	
  \pgfsetplotmarksize{0pt}

\begin{tikzpicture}
\begin{axis}[
	x=0.9cm,y=0.6cm,
    axis lines = middle,
    xmin=-0.2, xmax=8,
    ymin=-0.2, ymax=5,
    xtick = \empty,% {1,4,7}%-2,-1,0,1,2},
    ytick = \empty ,%{0,1,2},
    xlabel = $n$,
    ylabel = {$\fm_n$},
    extra x ticks ={1,4,6.5},
    extra x tick labels={$n_-$,$n_+$,$n_0$},
    extra y ticks ={1.2,1.65,2.4},
    extra y tick labels={$K$,$C_+n_-^{\alpha_+}$,$C_-n_+^{\alpha_-}$},
%     extra y tick labels={$C_+n_-^{\frac{r_+}{r_++1}}$,$C_-n_+^{\frac{r_-}{r_-+1}}$},
%     extra y ticks ={1.},
%     extra y tick labels={$K$},
]

\addplot[ domain= 1:4,
 		samples = 31, 
 		color = black, 
 		line width = 0.2mm,
		 dashed
 		]
 		{0.6 *x +0};
\addplot[ domain= 1:4,
		 samples = 31,
		 color = black,
		 line width = 0.2mm,
		 dashed
		 ]{0.9 *x + 0.75};
\addplot[ domain= 6.5:8,
		 samples = 11,
		 color = black,
		 line width = 0.2mm,
		 dashed
		 ]{1.2};
\addplot[ domain = 0:1, samples = 2, color = black, line width = 0.05mm]{1.65};
\addplot[ domain = 0:4, samples = 2, color = black, line width = 0.05mm]{2.4};

\addplot[line width=0.4mm, color= blue] table[x index=0,y index=1, col sep=comma] {datax.dat};
 
\end{axis}
\end{tikzpicture}

%(exp(-2*x)/3 -1 + x - exp(x-5))*(1 - 1/(1+ exp(-5(x-6)) )) + 1
  \end{center}
 \caption{\it Schematically illustrating the modulation curve $n\mapsto\fm_n$ for FSS on the circle. \label{fig:circular-modulation-scheme} Along $[n_-, n_+]$ the curve is between the lower ($C_- n^{\alpha_-}$) and upper ($C_+ n^{\alpha_+}$) bounds (dashed), satisfying the condition $C_+ n_-^{\alpha_+}\leq C_- n_+^{\alpha_-}$, and for $n\geq n_0$ it is below the horizontal upper bound (dashed). %The exponents relate to power smeariness $0<r_-<r_+$ via $\alpha_-=\frac{r_-}{r_-+1}$ and $\alpha_+=\frac{r_+}{r_++1}$ as detailed in Section \ref{scn:smeary} and Definition \ref{def:FSS}.
 }
\end{figure}
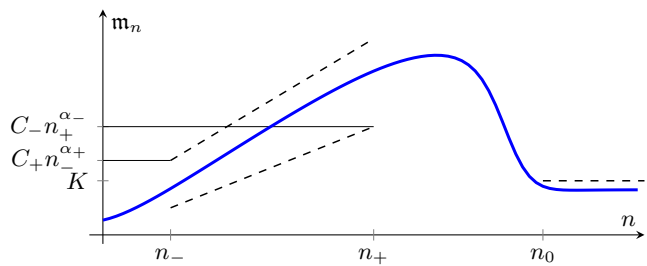
 Although under FSS, $\fm_n$ is eventually constant, i.e. the asymptotic rate of $\mun$ is the classical $n^{-1/2}$, for nonvanishing intervals of sample sizes $[n_-,n_+]$, the ``finite sample'' rate is (in expectation) between 
 $$ \Big(n^{-\frac{1}{2}} < \Big) \quad n^{-\frac{1-\alpha_-}{2}}\mbox{ and }n^{-\frac{1-\alpha_+}{2}}\,,$$
 i.e. like a smeary rate, cf. \cite{HundrieserEltznerHuckemann2020}.    
 
 Of course, as illustrated in Figure \ref{fig:circular-modulation-curves}, the modulation curve can be subject to different regimes of $\alpha_-$ and $\alpha_+$, in applications, typically the first regime is of interest, cf. Section \ref{scn:testing}. 
 
 \section{Correcting for Finite Sample Smeariness in Statistical Testing}\label{scn:testing}
 
The central limit theorem by \cite{HL98} and \cite{BP05} for an $m$-dimensional manifold $M$, cf. also \cite{H_Procrustes_10,H_ziez_geod_10,BL17} for sample Fr\'echet means $\mun$, has been extended by \cite{EltznerHuckemann2019} to random variables no longer avoiding arbitrary neighborhoods of possible cut points of the Fr\'echet mean $\mu$. Under nonsmeariness it has the following form:
\[ \sqrt{n}\, \phi(\mun) \inD \cN\left(0, 4\,H^{-1} \Sigma H^{-1}\right)\,. \]
\begin{figure}[b!]
  \centering
  \includegraphics[width = \textwidth, trim = 0 0 0 0, clip]{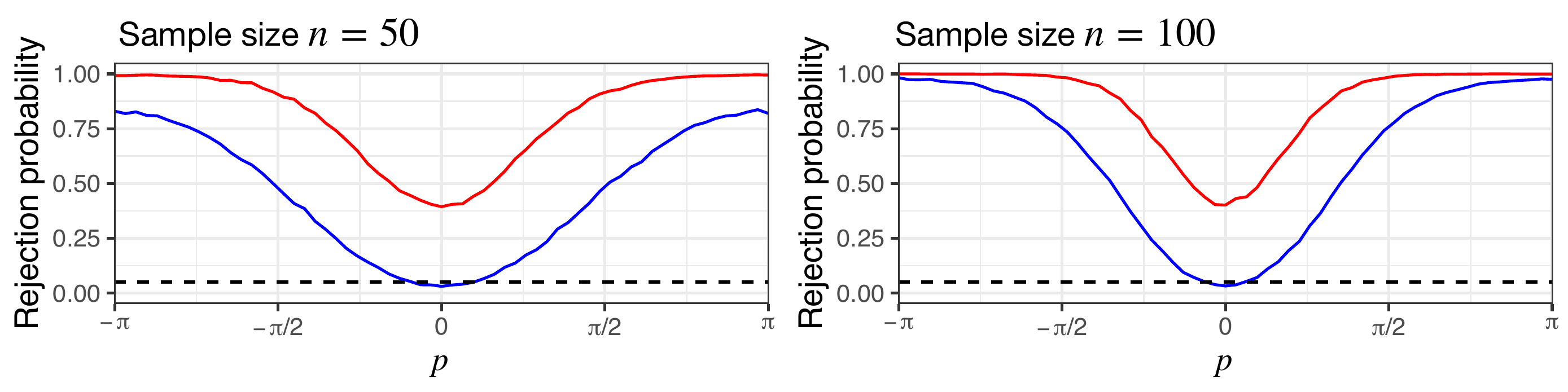}
  \caption{\it Empirical rejection probabilities of quantile based tests (red) and bootstrap based tests (blue) to test for significance $95\%$ if two samples of size $n = 50$ (left) and $n = 100$ (right) have identical Fr\'echet means. The two samples are taken independently from a von Mises distribution with mean $\mu=0$ and $\mu = p$, respectively, and concentration $\kappa = 1/2$. The dashed line represents $5\%$. }
  \label{fig:RejectionCurves}
\end{figure}
Here $\phi$ is a local chart mapping $\mu$ to the origin, $H$ is the expected value of the Hessian of the Fr\'echet function $F$ from (\ref{eq:Frechet-fcns}) in that chart at $\mu$ and $\Sigma$ is the covariance of $\phi(X)$. In practical applications, $H$ is usually ignored, as it has got no straightforward plugin estimators, and $4 H^{-1} \Sigma H^{-1}$ is simply estimated by the empirical covariance $\widehat{\Sigma}_n$ of $\phi(X_1),\ldots,\phi(X_n)$
giving rise to the approximation
\begin{eqnarray}\label{eq:BP-test} n \phi(\mun)^T\widehat{\Sigma}_n^{-1}\phi(\mun)&\inD& \chi^2_m\,,\end{eqnarray}
e.g. \cite{BP05,BL17}. For finite samples sizes, this approximation depends crucially on 
$$\fm_n = \frac{\EE[n\|\phi(\mun)\|^2]}{\EE[\tr(\widehat{\Sigma}_n)]}=1\,,$$
%  $$\fm_n = \frac{\EE[n\|\phi(\mun)\|^2]}{\frac{n-1}{n}\EE[\tr(\widehat{\Sigma}_n)]}=1\,,$$
and it is bad in regimes whenever $\fm_n\gg1$. 
This is illustrated in Figure~\ref{fig:RejectionCurves} where two samples from von Mises distributions with concentration $\kappa = 1/2$ are tested for equality of Fr\'echet means. 
%  Note by Figure \ref{fig:circular-modulation-curves} (left) that $\fm_n\geq 5$ for $n \geq 10$ thus indicating the presence of FSS.
Indeed the quantile based test does not keep the nominal level, whereas the bootstrap based test, see \cite{EH2017}, keeps the level fairly well and is shown to be consistent under FSS on $\SSo$, cf. \cite{HundrieserEltznerHuckemann2020}.
 
 Moreover, Table \ref{tab:wind} shows a comparison of $p$-values of the quantile test based on (\ref{eq:BP-test}) and the suitably designed bootstrap test for daily wind directions taken at Basel for the years 2018, 2019, and 2020, cf. Figure \ref{fig:wind}. While the quantile based test asserts that the year 2018 is high significantly different from 2019 and 2020, the bootstrap based test shows that a significant difference can be asserted at most for the comparison between 2018 and 2019.
%  all three periods are high significantly different, the bootstrap test shows that a significant change can be asserted at most from the first to the last time interval.
	The reason for the difference in $p$-values between quantile and bootstrap based test is the presence of FSS in the data, i.e. $\fm_n \gg 1$.
   Indeed, estimating for $n = 365$ the modulation $\fm_n$ of the yearly data using $B = 10.000$ bootstrap repetitions, as further detailed in \cite{HundrieserEltznerHuckemann2020}, yields $\fm_n^{2018} = 2.99$, $\fm_n^{2019} = 2.97$, and $\fm_n^{2020} = 4.08$.
%   
%   
%    ({\bf Shayan: do we have $n$ and the sup? $n = 365$, $\fm_n^{2018} = 2.99$, $\fm_n^{2019} = 2.97$, and $\fm_n^{2020} = 4.08$ }) .
 
  \begin{table}[t!]
  \centering
 \begin{tabular}{r||c|c|c} $p$-value\;&\;2018 vs. 2019\;  & \;2019 vs. 2020 \;&\;  2018 vs. 2020\\ \hline
%  \begin{tabular}{r||c|c|c} Test &90 -- 94 vs. 00 -- 04 & 00 -- 04 vs. 10 -- 14& 90 -- 94 vs. 10 -- 14 \\ \hline
  quantile based test\;& $0.00071$ &$0.27$&$0.019$\\
  bootstrap based test\;& $0.047$\textcolor{white}{00}&$0.59$ &$0.21$\textcolor{white}{0}
  \end{tabular} 
  \vspace*{0.5cm}
  \caption{\it Comparing $p$-values of the quantile based test for equality of means of yearly wind data from Basel (Figure \ref{fig:wind}), based on (\ref{eq:BP-test}) with the bootstrap test amending for FSS proposed in \cite{HundrieserEltznerHuckemann2020} for $B= 10.000$ bootstrap realizations. \label{tab:wind}}
  \end{table}
 
 \begin{figure}[t!]
 \centering
\includegraphics[width = \textwidth, trim = 0 20 0 5, clip]{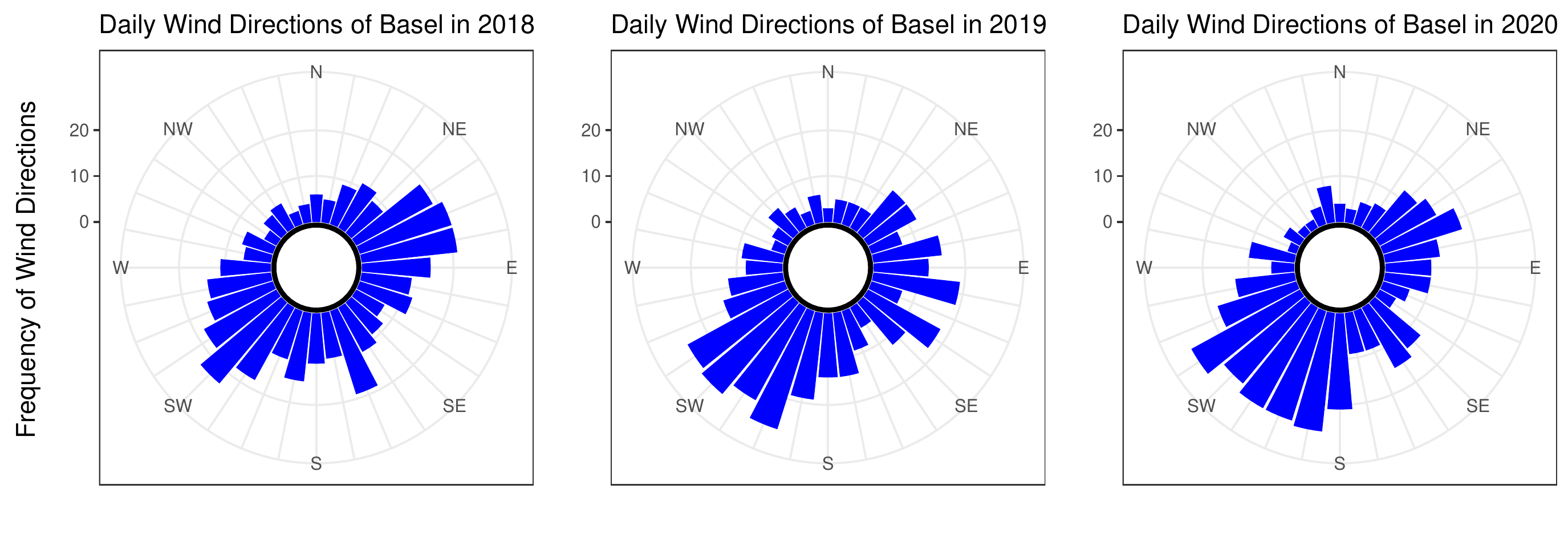}  
\caption{\it Histograms of daily wind directions for Basel (provided by meteoblue AG) for 2018 (left), 2019 (center), and 2020 (right).}
 \label{fig:wind}
 \end{figure}

%  In addition to having the wrong rate, due to $\fm_n> 1$, for distributions with nonvanishing density at the antipodal of the Fr\'echet mean, the tangent space covariance underestimates the actual covariance of rescaled Fr\'echet means (FSS of Type 1). It seems that FSS of Type 1 is typical for higher dimensional spheres.

% \begin{As}\label{As:2}
% In order to reduce notational complexity, we also assume that 
% \begin{eqnarray}\label{eq:Frecht-fcn}
%  f(x) &=&  \sum_{j=1}^m \limits T_j |(R x)_j|^r + o (|x|^r)
% \end{eqnarray}
% with some $r\geq 2$, where $(Rx)_j$ is the $j$-th component after multiplication with an orthogonal matrix $R$ and $T_1,\ldots,T_m$ are positive.
% \end{As}
% 
% With these definitions, we can define various asymptotic regimes.
% 
% \begin{Def}
%  We say that $X$ is
%  \begin{itemize}
%   \item[(i)] \emph{Euclidean} if $\fm_n = 1$ for all $n \in \NN$,
%   \item[(ii)] \emph{finite sample smeary} if $1 <\sup_{n\in\NN} \fm_n < \infty$.
%   \item[(iii)] \emph{smeary} if $\sup_{n\in\NN} \fm_n = \infty$,
%   \item[(iv)] \emph{$r$-power smeary} if (\ref{eq:Frecht-fcn}) holds with $r>2$.
%  \end{itemize}
% \end{Def}

 \section{Finite Sample Smeariness Universality}
 
 Consider  $p \in \SSS$ parametrized as $(\theta, \sin\theta q) \in \SSS$ where $\theta \in [0,\pi]$ denotes distance from the north pole $\mu \in \SSS$ and $q \in \mathbb{S}^{m-1}$, which is rescaled by $\sin\theta$.  
 
\begin{Th}
 Let $m \ge 4$, $Y$ uniformly distributed on $\mathbb{S}^{m-1}$ and $K>1$ arbitrary. Then there are $\theta^* \in (\pi/2,\pi)$ and $\alpha \in (0,1)$ such that for every $\theta \in (\theta^*,\pi)$ a random variable $X$ on $\SSS$ with $\Prb\{X=(\theta, \sin \theta\, Y)\}= \alpha$ and $\Prb\{X=\mu\}= 1 - \alpha$ features
%   \begin{align*}
$  \sup_{n\in \NN} \limits \fm_n \ge \lim_{n \to \infty} \limits \fm_n > K.$
%   \end{align*} 
  In particular,  
  $\theta^*=\frac{\pi}{2} + \mathcal{O}(m^{-1})\,.$
\end{Th}

 \begin{proof} The first assertion follows from \cite[Theorem 4.3]{Eltzner2020a} and its proof in Appendix A.5 there. Notably $\theta^* = \theta_{m,4}$ there. The second assertion has been shown in \cite[Lemma A.5]{Eltzner2020a}.\end{proof}

\begin{Th}
  Let $X$ be a random variable on $\mathbb{S}^m$ with $m \ge 2$ with unique nonsmeary mean $\mu$, which is invariant under rotation around $\mu$ and which is not a point mass at $\mu$. Then $\mu$ is Type I finite sample smeary.
\end{Th}

\begin{proof}
  From \cite{Eltzner2020a}, page 17, we see that the Fr\'echet function $F_\theta$ for a uniform distribution on the $\mathbb{S}^{m-1}$ at polar angle $\theta$ evaluated at a point with polar angle $\psi$ from the north pole is
  \begin{align*}
    a(\psi, \theta, \phi) :=& \arccos\left( \cos\psi \cos\theta + \sin\psi \sin\theta \cos\phi \right)\\
    F_\theta (\psi) =& \left( \int_0^{2\pi} \sin^{m-2}\phi \, d\phi \right)^{-1} \int_0^{2\pi} \sin^{m-2}\phi \, a^2(\psi, \theta, \phi) \, d\phi \, .
  \end{align*}
  Defining a probability measure $d\mathbb{P}(\theta)$ on $[0,\pi]$, the Fr\'echet function for the corresponding rotation invariant random variable is
  \begin{align*}
    F (\psi) = \int_0^\pi F_\theta (\psi) \, d\mathbb{P}(\theta) \, .
  \end{align*}
  On page 17 of \cite{Eltzner2020a} the function
  \begin{align*}
    f_2(\theta,\psi) := \frac{1}{2} \sin^{m-1} \theta \int_0^{2\pi} \sin^{m-2}\phi \, d\phi \, \frac{d^2}{d\psi^2} F_\theta (\psi)
  \end{align*}
  is defined and from Equation (5) on page 19 we can calculate
  \begin{align*}
  f_2(\theta,0) =& \sin^{m-2} \theta \left( \frac{1}{m-1} \sin\theta + \theta \cos\theta \right) \int_0^{2\pi} \sin^{m}\phi \, d\phi\\
  =& \sin^{m-1} \theta \int_0^{2\pi} \sin^{m-2}\phi \, d\phi \left( \frac{1}{m} + \frac{m-1}{m} \theta \cot\theta \right)\, , 
  \end{align*}
  which yields the Hessian of the Fr\'echet function for $d\mathbb{P}(\theta)$ as
  \begin{align*}
    \textnormal{Hess}F(0) &= 2 \text{Id}_m \int_0^\pi \left( \frac{1}{m} + \frac{m-1}{m} \theta \cot \theta \right) d\mathbb{P}(\theta) \, .
  \end{align*}
  One sees that $\theta \cot \theta \le 0$ for $\theta \ge \pi/2$. For $\theta \in (0, \pi/2)$ we have
  \begin{align*}
    \tan \theta > \theta \quad \Leftrightarrow \quad \theta \cot \theta < 1 \quad \Leftrightarrow \quad \left( \frac{1}{m} + \frac{m-1}{m} \theta \cot \theta \right) < 1 \, .
  \end{align*}
  Using $\Sigma[\mu]$ to denote the CLT limit $n \textnormal{Cov}[\widehat{\mu}_n] \to \Sigma[\mu]$, cf. \cite{BP05}, we get the result
  \begin{align*}
    \textnormal{Hess}F(\mu) < 2 \text{Id}_m \quad \Rightarrow \quad \Sigma[\mu] > \textnormal{Cov} \left[\log_\mu X \right] \quad \Rightarrow \quad \tr \left( \Sigma[\mu] \right) > \textnormal{Var}[X] \, .
  \end{align*}
  The claim follows at once.
\end{proof}

\begin{Conj}
  Let $X$ be a random variable supported on a set $A \subset \SSS$ whose convex closure has nonzero volume and which has a unique mean $\mu$. Then $\mu$ is Type I finite sample smeary.
\end{Conj}

\begingroup
\let\clearpage\relax

\vspace*{\baselineskip}

\endgroup

\end{document}